\documentclass[a4paper, reqno, 12pt]{amsart}

\usepackage[utf8]{inputenc} 
\usepackage[T1, T2A]{fontenc}
\usepackage[english]{babel}
\usepackage{amsmath, amssymb, amsfonts, amsthm, amscd}
\usepackage{enumitem}
\usepackage{tikz}
\usepackage[hidelinks]{hyperref}
\usetikzlibrary{matrix}
\usetikzlibrary{calc}
\usepackage{dutchcal}

\setlist[enumerate]{nolistsep, label = (\alph*), ref=(\text{\alph*)}}
\setlist[itemize]{itemsep=3pt}

\renewcommand{\Im}{\mathop{\text{Im}}}
\renewcommand{\phi}{\varphi}
\renewcommand{\ge}{\geqslant}
\renewcommand{\le}{\leqslant}

\newcommand{\K}{\mathbb{K}}
\newcommand{\A}{\mathbb{A}}
\newcommand{\Q}{\mathbb{Q}}
\newcommand{\Z}{\mathbb{Z}}
\newcommand{\N}{\mathbb{Z}_{>0}}
\newcommand{\G}{\mathbb{G}}
\newcommand{\Zgezero}{\mathbb{Z}_{\geqslant 0}}

\newcommand{\g}{\mathcal{g}}
\newcommand{\RR}{R(\g)}
\newcommand{\frK}{\mathfrak{K}}
\newcommand{\divid}{\mid}
\newcommand{\rep}{\widehat}

\newcommand{\dCb}{\delta_{C, \beta}}
\newcommand{\lici}{{l_{ic_i}}}

\DeclareMathOperator{\Ker}{Ker}
\DeclareMathOperator{\rdeg}{\rep \deg}

\theoremstyle{plain}
\newtheorem{lemma}{Lemma}
\newtheorem{proposition}{Proposition}
\newtheorem{theorem}{Theorem}
\newtheorem{corollary}{Corollary}
\theoremstyle{definition}
\newtheorem*{definition}{Definition}
\newtheorem{construction}{Construction}
\newtheorem{example}{Example}
\newtheorem{conjecture}{Conjecture}
\newtheorem{problem}{Problem}
\theoremstyle{remark}

\begin{document}

\title[Homogeneous derivations of trinomial algebras]{Homogeneous locally nilpotent derivations of non-factorial trinomial algebras}
\author{Yulia~Zaitseva}
\date{}
\address{Lomonosov Moscow State University, Faculty of Mechanics and Mathematics, 
Department of Higher Algebra, Leninskie Gory~1, Moscow, 119991 Russia}
\email{yuliazaitseva@gmail.com}
\thanks{The work was supported by the Foundation for the Advancement of Theoretical Physics and Mathematics ``BASIS''}

\begin{abstract}
\noindent
We describe homogeneous locally nilpotent derivations of the algebra of regular functions for a class of affine trinomial hypersurfaces. This class comprises all non-factorial trinomial hypersurfaces.
\end{abstract}

\maketitle

\section{Introduction}

Let $\K$ be an algebraically closed field of characteristic zero and $R$ an algebra over~$\K$. 
A~derivation of~$R$ is a $\K$-linear map $\delta\colon R \to R$ satisfying the Leibniz rule: 
\hbox{$\delta(fg) = \delta(f)g + f\delta(g)$} for all~$f, g \in R$. 
A derivation $\delta$ is called locally nilpotent if for every $f \in R$ there is an $m \in \N$ such that $\delta^m(f) = 0$. 

Let $X$ be an irreducible affine algebraic variety over~$\K$ and $\G_a = (\K, +)$ the additive group of the ground field. 
Locally nilpotent derivations of the algebra $\K[X]$ are known to be in one-to-one correspondence with regular $\G_a$-actions on the variety~$X$, see e.g.~\cite[Section~1.5]{Fr2006}. 

Suppose that $R$ is graded by a finitely generated abelian group~$K$: $$R = \bigoplus \limits_{w \in K} R_w.$$
A derivation $\delta \colon R \to R$ is said to be homogeneous if it maps homogeneous elements to homogeneous ones. 
In such a case, there exists an element $\deg \delta \in K$ satisfying \hbox{$\delta(R_w) \subseteq R_{w + \deg \delta}$} for all $w \in K$. 
The element $\deg \delta \in K$ is called the degree of~$\delta$. 

Recall that an affine algebraic group is called a quasitorus if it is isomorphic to the direct product of a torus and a finite abelian group. 
Suppose that a quasitorus~$H$ acts on the variety~$X$. 
Such an action corresponds to a grading on the algebra~$\K[X]$ by the group of characters~$K$ of the quasitorus~$H$: 
$$
\K[X] = \bigoplus \limits_{w \in K} \K[X]_w, \quad \text{where}\quad \K[X]_w = \{f \mid h\circ f = w(h)f \quad\forall h \in H\}.
$$
It is easily shown that a locally nilpotent derivation of~$\K[X]$ is homogeneous with respect to this grading 
if and only if the quasitorus~$H$ normalizes the corresponding $\G_a$-action on~$X$. 
A description of homogeneous locally nilpotent derivations enables us to describe the automorphism group of an algebraic variety, see e.g.~\cite[Theorem~5.5]{ArHaHeLi2012}, \cite[Theorem~5.5]{ArGa2017}. 

Fix positive integers $n_0, n_1$, $n_2$ and let $n = n_0 + n_1 + n_2$. 
For each $i = 0, 1, 2$, fix a tuple $l_i = (l_{ij} \mid j = 1, \ldots, n_i)$ of positive integers $l_{ij}$ 
and define a monomial $T_i^{l_i} = T_{i1}^{l_{i1}} \ldots T_{in_i}^{l_{in_i}}$ 
in the polynomial algebra $\K[T_{ij},\, i = 0, 1, 2,\, j = 1, \ldots, n_i]$. 
By a trinomial we mean a polynomial 
$$
\g = T_0^{l_0} + T_1^{l_1} + T_2^{l_2}.
$$
A trinomial hypersurface $X(\g)$ is the zero set $\{\g = 0\}$ in the affine space~$\A^n$. 
By~$\RR = \K[X(\g)]$ we denote the algebra of regular functions on $X(\g)$. 

Our motivation to study trinomials comes from toric geometry. 
Consider an effective action ${T \times X \to X}$ of an algebraic torus~$T$ on an irreducible variety~$X$. 
The complexity of such an action is the codimension of a general $T$-orbit in $X$. It equals $\dim X - \dim T$. 

Actions of complexity zero are actions with an open $T$-orbit. 
A normal variety admitting a torus action with an open orbit is called a toric variety. 
If $X$ is a toric (not necessary affine) variety with the acting torus~$T$, then $\G_a$-actions on~$X$ normalized by~$T$ can be described in terms of Demazure roots of the fan corresponding to~$X$; see~\cite{De1970}, \cite[Section~3.4]{Od1988} for the original approach and~\cite{Li22010, ArLi2012, ArHaHeLi2012} for generalizations. 

Let $T \times X \to X$ be a torus action of complexity one. 
A description of $\G_a$-actions on~$X$ normalized by~$T$ in terms of proper polyhedral divisors 
may be found in~\cite{Li12010} and~\cite{Li22010}. 
It is an interesting problem to find their explicit form in concrete cases. 

The study of toric varieties is related to binomials, see e.g.~\cite[Chapter~4]{St1996}. 
At the same time, Cox rings establish a close relation between torus actions of complexity one and
trinomials, see~\cite{HaSu2010, HaHeSu2011, HaHe2013, ArHaHeLi2012, HaWr2017}. 
In particular, any trinomial hypersurface admits a torus action of complexity one.

In this paper we study locally nilpotent derivations of~$\RR$ that are homogeneous with respect to the ''finest'' grading. 
This grading is defined in Construction~\ref{gr_constr}. 
The ''finest'' grading corresponds to an effective action of a quasitorus~$H$ on~$X(\g)$. 
The~neutral component of the quasitorus~$H$ is a torus of dimension~$n - 2$. It acts on~$X(\g)$ with complexity one. 

The weight monoid of a graded algebra $R = \bigoplus \limits_{w \in K} R_w$ is the set ${S = \{w \in K \mid R_w \ne 0\}}$. 
The weight cone of $R$ is the convex cone $\omega$ in the rational vector space ${K_\Q = K \otimes_\Z \Q}$
generated by the weight monoid $S$. 
A homogeneous derivation is said to be primitive if its degree does not lie in the weight cone $\omega \subseteq K_\Q$. 
It is clear that every primitive derivation is locally nilpotent. 
The converse is false, see Example~\ref{ex_th_1332}. 
An explicit description of all primitive derivations of~$\RR$ that are homogeneous with respect to the ''finest'' grading is given in~\cite[Theorem~4.4]{ArHaHeLi2012}\footnote{In~\cite{ArHaHeLi2012}, a more general class of affine varieties defined by a system of trinomials is studied, see~\cite[Construction~3.1]{ArHaHeLi2012} for details.}. 
Such derivations have the form $h \dCb$, where $\dCb$ are some special primitive derivations of~$\RR$ (see~\cite[Construction~4.3]{ArHaHeLi2012} or Construction~\ref{dCb_constr} below), and $h$ is a homogeneous element in the kernel of~$\dCb$. 
We call derivations of the form $h \dCb$ elementary. 

We expect that all homogeneous locally nilpotent derivations of trinomial algebras are elementary (see Conjecture~\ref{conj} in Section~\ref{open}). 
We verify this conjecture in Theorem~\ref{th_01m} for some types of trinomials. 
In particular, the conjecture is proved for all non-factorial trinomial hypersurfaces (Corollary~\ref{cor_1}). 
Corollary~\ref{cor_2} gives a criterion for existence of homogeneous locally nilpotent derivations of trinomial algebras. 
It enables us to give a new proof of criteria for rigidity of factorial trinomial hypersurfaces, which was proved earlier in~\cite[Theorem~1]{Ar2016} (see Corollary~\ref{cor_3}).

\smallskip

The author is grateful to her supervisor Ivan~Arzhantsev for posing the problem and permanent support and to Sergey~Gaifullin for useful discussions and comments.

\section{Preliminaries}

In this section the ''finest'' $\deg$-grading and elementary derivations of a trinomial algebra are defined. 
\begin{construction}
\label{gr_constr}
Fix positive integers $n_0$, $n_1$, $n_2$ and let $n = n_0 + n_1 + n_2$. 
For each $i = 0, 1, 2$, fix a tuple of positive integers $l_i = (l_{ij} \mid j = 1, \ldots, n_i)$ and 
define a monomial $T_i^{l_i} = T_{i1}^{l_{i1}} \ldots T_{in_i}^{l_{in_i}} \in \K[T_{ij},\, i = 0, 1, 2,\, j = 1, \ldots, n_i]$.  
By a \textit{trinomial} we mean a polynomial of the form 
$$
\g = T_0^{l_0} + T_1^{l_1} + T_2^{l_2}.
$$
A \textit{trinomial hypersurface} $X(\g)$ is the zero set $\{\g = 0\}$ in the affine space $\A^n$. 
It can be checked that the polynomial $\g$ is irreducible, hence the 
algebra $\RR := \K[T_{ij}]\,/\,(\g)$ of regular functions on $X(\g)$ has no zero divisors. 
We call such algebras $\RR$ \textit{trinomial}.  
We~use the same notation for elements of $\K[T_{ij}]$ and their projections to~$\RR$.

Following~\cite{HaHe2013}, we build a $2 \times n$ matrix $L$ from the trinomial $\g$ as follows: 
$$L = 
\begin{pmatrix}
-l_0 & l_1 & 0 \\
-l_0 & 0 & l_2 
\end{pmatrix}.
$$
Let $L^*$ be the transpose of~$L$. 
Denote by $K$ the factor group $K = \Z^n\,/\,\Im L^*$ and by $Q \colon \Z^n \to K$ the projection. 
Let $e_{ij} \in \Z^n$, $i = 0, 1, 2$, $j = 1, \ldots, n_i$, be the canonical basis vectors. 
The equalities 
\begin{equation}
\label{def_deg}
\deg T_{ij} = Q(e_{ij})
\end{equation}
define a $K$-grading on the algebra~$\K[T_{ij}]$.

Since $\g$ is a homogeneous polynomial of degree
$$\mu = l_{i1} Q(e_{i1}) + \ldots + l_{in_i} Q(e_{in_i})$$
for some $\mu \in K$ and any $i = 0, 1, 2$, it follows that  equalities~(\ref{def_deg}) define a $K$-grading on~$\RR = \K[T_{ij}]\,/\,(\g)$ as well. 
By ''$\deg$'' we denote the degree with respect to this grading. 
\end{construction}

\begin{example}
\label{ex_gr_1332}
Let $\g = T_{01} T_{02}^3 + T_{11}^3 + T_{21}^2$, i.e., 
$L = \begin{pmatrix} -1 & -3 & 3 & 0 \\ -1 & -3 & 0 & 2 \end{pmatrix}$ (see~\cite[Example 3.3]{ArHaHeLi2012}). 

Since the matrix~$L$ can be reduced by integer elementary row and column operations to the form $\begin{pmatrix} 1 & 0 & 0 & 0 \\ 0 & 1 & 0 & 0 \end{pmatrix}$, it follows that the grading group $K = \Z^4\,/\,\Im L^*$ is isomorphic to~$\Z^2$. 
The grading can be given explicitly via
$$
\deg T_{01} = \begin{pmatrix} -3 \\ 3 \end{pmatrix}, \quad
\deg T_{02} = \begin{pmatrix} 1 \\ 1 \end{pmatrix}, \quad
\deg T_{11} = \begin{pmatrix} 0 \\ 2 \end{pmatrix}, \quad
\deg T_{21} = \begin{pmatrix} 0 \\ 3 \end{pmatrix}.
$$
It is unique up to the choice of basis in $\Z^4$. 
\end{example}

\begin{example}
\label{ex_gr_11112}
Let $\g = T_{01}T_{02} + T_{11}T_{12} + T_{21}^2$ and 
$L = \begin{pmatrix} -1 & -1 & 1 & 1 & 0 \\ -1 & -1 & 0 & 0 & 2 \end{pmatrix}$. 

The matrix $L$ can be reduced by elementary row and column operations to the form $\begin{pmatrix} 1 & 0 & 0 & 0 & 0 \\ 0 & 1 & 0 & 0 & 0 \end{pmatrix}$, 
hence the group $K = \Z^5\,/\,\Im L^*$ is isomorphic to $\Z^3$. 
We can define the grading explicitly via
$$
\deg T_{01} = \begin{pmatrix} 2 \\ 0 \\ 1\end{pmatrix}, \;
\deg T_{02} = \begin{pmatrix} 0 \\ 2 \\ -1\end{pmatrix}, \;
\deg T_{11} = \begin{pmatrix} 2 \\ 2 \\ 1\end{pmatrix}, \;
\deg T_{12} = \begin{pmatrix} 0 \\ 0 \\ -1\end{pmatrix}, \;
\deg T_{21} = \begin{pmatrix} 1 \\ 1 \\ 0\end{pmatrix}.
$$
\end{example}

\begin{example}
\label{ex_gr_222}
Let $\g = T_{01}^2 + T_{11}^2 + T_{21}^2$ and 
$L = \begin{pmatrix} -2 & 2 & 0 \\ -2 & 0 & 2 \end{pmatrix}$. 

Since $L$ can be reduced to the form $\begin{pmatrix} 2 & 0 & 0 \\ 0 & 2 & 0 \end{pmatrix}$, 
it follows that $K = \Z^3\,/\,\Im L^*$ is isomorphic to $\Z \oplus \Z_2 \oplus \Z_2$. 
Denote $\Z_2 = \{[0]_2, [1]_2\}$; the grading can be given explicitly via
$$
\deg T_{01} = \begin{pmatrix} 1 \\ [0]_2 \\ [0]_2 \end{pmatrix}, \quad
\deg T_{11} = \begin{pmatrix} 1 \\ [1]_2 \\ [0]_2 \end{pmatrix}, \quad
\deg T_{21} = \begin{pmatrix} 1 \\ [0]_2 \\ [1]_2 \end{pmatrix}.
$$
\end{example}

The following lemma shows that the constructed $\deg$-grading is the ''finest'' grading on the algebra~$\RR$ such that all generators~$T_{ij}$ are homogeneous. 

\begin{lemma}
\label{deg_lem}
Let $\deg$ be the $K$-grading on the algebra~$\RR$ defined in Construction~\ref{gr_constr} and $\rdeg$ be any $\rep K$-grading on~$\RR$ by some abelian group~$\rep K$ such that all generators~$T_{ij}$ are homogeneous. 
Then $\rdeg = \psi \circ \deg$ for some homomorphism $\psi \colon K \to \rep K$. In particular, any derivation that is homogeneous with respect to $\deg$-grading is homogeneous with respect to $\rdeg$-grading too. 
\end{lemma}

\begin{proof}
The weights $\rep w_{ij} = \rdeg T_{ij}$ determine a well-defined grading on the algebra~$\RR$ if and only if the polynomial~$\g$ is homogeneous with respect to the $\rdeg$-grading, i.e., the sums $l_{i1}\rep w_{i1} + \ldots + l_{in_i}\rep w_{in_i} \in \rep K$ are the same for all $i = 0, 1, 2$. 
The factorization of $\Z^n = \langle e_{ij}\rangle$ by $\Im L^*$ provides these equalities for the images of $e_{ij}$ in~$K$. 
Any other grading on~$\RR$ such that generators~$T_{ij}$ are homogeneous is obtained by further factorization of the group $K = \Z^n\,/\,\Im L^*$, and this factorization is the required map~$\psi$. 
\end{proof}

The following construction is given in~\cite{ArHaHeLi2012} and is described below in the case of trinomial hypersurfaces 
(in notation of~\cite{ArHaHeLi2012} that is $r = 2$, 
$A = \left(\begin{smallmatrix} 0 & -1 & 1 \\ 1 & -1 & 0 \end{smallmatrix}\right)$, 
$\g = g_I$ for $I = \{0, 1, 2\}$, $\RR = R(A, P_0)$).

\begin{construction}
\label{dCb_constr} 
Let us define a derivation $\dCb$ of~$\RR$. 
The input data are
\begin{itemize}
\item a sequence $C = (c_0, c_1, c_2)$, where $c_i \in \Z$, $1 \le c_i \le n_i$; 
\item a vector $\beta = (\beta_0, \beta_1, \beta_2)$ such that $\beta_i \in \K$, $\beta_0 + \beta_1 + \beta_2 = 0$. 
\end{itemize}
It is clear that for $\beta \ne 0$ as above either all entries $\beta_i$ differ from zero or there is a unique~$i_0$ with $\beta_{i_0} = 0$. 
According to these cases, we put further conditions and define: 
\begin{enumerate}[label = (\roman*), ref=(\roman*)]
\item \label{dCbconstr_1}
if $\beta_i \ne 0$ for all $i = 0, 1, 2$ and there is at most one $i_1$ with $l_{i_1c_{i_1}} > 1$, then we set 
$$
\dCb(T_{ij}) = 
\begin{cases}
\beta_i \prod \limits_{k \ne i} \frac{\partial T_k^{l_k}}{\partial T_{kc_k}},    &j = c_i,\\
0, &j \ne c_i;
\end{cases}
$$ 
\item \label{dCbconstr_2}
if $\beta_{i_0} = 0$ for a unique $i_0$ and there is at most one $i_1$ with $i_1 \ne i_0$ and $l_{i_1c_{i_1}} > 1$, then 
$$
\dCb(T_{ij}) = 
\begin{cases}
\beta_i \prod \limits_{k \ne i, i_0} \frac{\partial T_k^{l_k}}{\partial T_{kc_k}},    &j = c_i,\\
0, &j \ne c_i.
\end{cases}
$$
\end{enumerate}
\end{construction}

These assignments define a map~$\dCb$ on generators~$T_{ij}$.
It can be extended uniquely to a~derivation on~$\K[T_{ij}]$ by Leibniz rule. 
Clearly, we have $\dCb(\g) = 0$, whence the constructed map induces a well-defined derivation of the factor algebra~$\RR$.

\begin{lemma}
Every derivation~$\dCb$ of the algebra~$\RR$ is primitive and locally nilpotent. 
\end{lemma}

The proof is given in~\cite[Construction~4.3]{ArHaHeLi2012}.

\smallskip

Let $h \in \RR$ be a homogeneous element in the kernel of a derivation $\dCb$. 
The derivation~$h \dCb$ is also locally nilpotent. 

\begin{definition}
We say that a derivation of a trinomial algebra~$\RR$ is \textit{elementary} if it has the form~$h \dCb$, where $h$ is a homogeneous element in the kernel of~$\dCb$. 
In addition, \textit{elementary derivations of Type~I} are elementary derivations with $\dCb$ corresponding to case~\ref{dCbconstr_1} in Construction~\ref{dCb_constr}; \textit{derivations of Type~II} are elementary derivations with $\dCb$ corresponding to case~\ref{dCbconstr_2}. 
\end{definition}

\begin{example}
\label{ex_dCb_1332}
Suppose $\g = T_{01} T_{02}^3 + T_{11}^3 + T_{21}^2$ (see Example~\ref{ex_gr_1332} and ~\cite[Example 4.7]{ArHaHeLi2012}). 
Let us find all elementary derivations of the algebra~$\RR$. 
Note that there is a unique variable~$T_{01}$ in $\g$ with exponent~$1$. Hence only case~\ref{dCbconstr_2} with $i_0 \ne 0$ and  $C = (1, 1, 1)$ is possible. 
Thus, we have two cases for elementary derivations of Type~II:
\begin{enumerate}
\item $i_0 = 1$, that is $\beta = (\beta_0, 0, -\beta_0)$ for some $\beta_0 \in \K$. Then
$$
\dCb(T_{01}) = 2\beta_0 T_{21}, \quad 
\dCb(T_{02}) = 0, \quad 
\dCb(T_{11}) = 0, \quad 
\dCb(T_{21}) = -\beta_0 T_{02}^3,
$$
i.e., $\dCb = 2\beta_0 T_{21} \dfrac{\partial}{\partial T_{01}} -\beta_0 T_{02}^3 \dfrac{\partial}{\partial T_{21}}$, $h \in \Ker \dCb$; 
\item $i_0 = 2$, that is $\beta = (\beta_0, -\beta_0, 0)$ for some $\beta_0 \in \K$. Then
$$
\dCb(T_{01}) = 3\beta_0 T_{11}^2, \quad 
\dCb(T_{02}) = 0, \quad 
\dCb(T_{11}) = -\beta_0 T_{02}^3, \quad 
\dCb(T_{21}) = 0,
$$
i.e., $\dCb = 3\beta_0 T_{11}^2 \dfrac{\partial}{\partial T_{01}} -\beta_0 T_{02}^3 \dfrac{\partial}{\partial T_{11}}$, $h \in \Ker \dCb$.
\end{enumerate}
\end{example}

\begin{example}
\label{ex_dCb_11112}
Let $\g = T_{01}T_{02} + T_{11}T_{12} + T_{21}^2$ (see Example~\ref{ex_gr_11112}). 
There is a unique variable $T_{ij}$ in the trinomial with exponent $l_{ij}$ such that $l_{ij} > 1$. 
Hence the algebra~$\RR$ admits elementary derivations of both types: 
we can take any sequence $C = (c_0, c_1, 1)$, $c_0, c_1 \in \{1, 2\}$, 
and any vector $\beta = (\beta_0, \beta_1, \beta_2)$ under the condition $\beta_0 + \beta_1 + \beta_2 = 0$. 
Let us give an example of elementary derivation of Type~I:
$$\dCb = T_{11}T_{21} \dfrac{\partial}{\partial T_{02}} + T_{01}T_{21} \dfrac{\partial}{\partial T_{12}} - 
T_{01}T_{11} \dfrac{\partial}{\partial T_{21}}.$$
\end{example}

\begin{example}
\label{ex_dCb_222}
Let $\g = T_{01}^2 + T_{11}^2 + T_{21}^2$ (see Example~\ref{ex_gr_222}); 
all exponents in this trinomial are greater than~$1$. Hence there exists no elementary derivation. 
\end{example}

\section{Auxiliary lemmas}
We use notation introduced in the previous section. 
Let us recall some definitions. Let~$K$ be an abelian group and $R$ a $K$-graded algebra. By a \textit{$K$-prime} element of~$R$ we mean a homogeneous nonzero nonunit $f \in R$ such that $f \divid gh$ with homogeneous $g, h \in R$ implies $f \divid g$ or $f \divid h$. 
We say that $R$ is \textit{factorially $K$-graded} if every nonzero homogeneous nonunit of~$R$ is a product of $K$-primes. 
It is clear that this decomposition is unique up to permutation of factors and multiplication by units. 

\begin{lemma} \label{div_lem}
With the notation of Construction~\ref{gr_constr}, the following statements hold. 
\begin{enumerate}
\item \label{TijKprime}
The generators~$T_{ij}$ are pairwise nonassociated $K$-prime elements of~$\RR$.
\item \label{fact_graded}
The algebra~$\RR$ is factorially $K$-graded. 
\end{enumerate}
\end{lemma}

This was proved in~\cite[Proposition 2.2(i) and Theorem 1.1(i)]{HaHe2013}.

\smallskip

The following two lemmas include basic properties of locally nilpotent derivations.  
\begin{lemma} 
\label{fr_lem}
Suppose $\delta\colon R \to R$ is a locally nilpotent derivation of a domain $R$ and $f, g \in R$. 
\begin{enumerate}
\item \label{fr_lem_f} 
If $f \divid \delta(f)$, then $\delta(f) = 0$.
\item \label{fr_lem_ker} 
The derivation $f \delta$ is locally nilpotent if and only if $f \in \Ker \delta$ holds. 
\item \label{fr_lem_fg}
If $f \divid \delta(g)$ and $g \divid \delta(f)$, then $\delta(f) = 0$ or $\delta(g) = 0$. 
\item \label{fr_lem_14}
If $\delta = \sum \limits_{m \le i \le n} \delta_i$, 
where all derivations $\delta_i$ are homogeneous with respect to some \hbox{$\Z$-grading} on~$R$,
$\delta_m, \delta_n \ne 0$, then $\delta_m$ and $\delta_n$ are locally nilpotent. 
\end{enumerate}
\end{lemma}

The proof can be found, for example, in~\cite[Principles 5, 7, 14 and Corollary 1.20]{Fr2006}. 

\begin{lemma}
\label{fact_hom_der_lem}
Let $R$ be a finitely generated $\Z^k$-graded domain. If there exists a nonzero locally nilpotent derivation of~$R$, then $R$ admits a nonzero homogeneous locally nilpotent derivation. 
\end{lemma}

\begin{proof}
Denote given locally nilpotent derivation on~$R$ by~$\delta^{(0)}$.
Consider the $\Z$-grading on the algebra~$R$ induced by the first component in~$\Z^k$.  
Since $R$ is finitely generated, we have $\delta^{(0)} = \sum \limits_{m \le i \le n} \delta^{(0)}_i$ for some homogeneous derivations~$\delta^{(0)}_i$. It follows from Lemma~\ref{fr_lem}\ref{fr_lem_14} that $\delta^{(1)} = \delta^{(0)}_n$ is a nonzero locally nilpotent derivation that is homogeneous with respect to the $\Z$-grading by the first component of~$\Z^k$. 
Applying Lemma~\ref{fr_lem}\ref{fr_lem_14} $k - 1$ times to the $\Z$-gradings by other components of~$\Z^k$, we obtain a derivation~$\delta^{(k)}$ that is homogeneous with respect to the grading by all components of~$\Z^k$, i.e., with respect to the $\Z^k$-grading. 
\end{proof}

The following lemma is proved in~\cite[Lemma~3.4]{Ga2017}. 
For convenience of the reader we give a short proof below. 

\begin{lemma}
\label{3_lemma}
Let $\delta$ be a $\deg$-homogeneous locally nilpotent derivation of an algebra~$\RR$. 
Then there exists at most one variable~$T_{ij}$ in every monomial~$T_i^{l_i}$ such that $\delta(T_{ij}) \ne 0$. 
\end{lemma}

\begin{proof}
Assume the converse. Then there is a monomial with at least two variables not belonging to the kernel of the derivation. 
We can assume that these variables are $T_{01}$ and $T_{02}$ in $T_0^{l_0}$.

Consider the following grading on the algebra~$\K[T_{ij}]$:
\begin{center}
\renewcommand{\arraystretch}{1.4}
\begin{tabular}{c ||   c|c|c| c |c    |    c| c |c     |    c| c |c}
$T_{ij}$ & 
$T_{01}$ & $T_{02}$ & $T_{03}$ & $\ldots$ & $T_{0n_0}$ & 
$T_{11}$ & $\ldots$ & $T_{1n_1}$ & 
$T_{21}$ & $\ldots$ & $T_{2n_2}$ \\
\hline
$\rdeg T_{ij}$ & 
$l_{02}$ & $-l_{01}$ & $0$ & $\ldots$ & $0$ &
$0$ & $\ldots$ & $0$ &
$0$ & $\ldots$ & $0$
\end{tabular}
\end{center}
The trinomial $\g$ is homogeneous (of degree $0$) with respect to this grading, therefore the \hbox{$\rdeg$-grading} is a well-defined grading on the factor algebra~$\RR$. 

By Lemma~\ref{deg_lem}, it follows that the derivation~$\delta$ is $\rdeg$-homogeneous. Then we have the following two cases.

1) $\rdeg \delta \ge 0$. 
Then $\rdeg \delta(T_{01}) = \rdeg \delta + \rdeg T_{01} > 0$. 
Note that $T_{01}$ is a unique variable with a positive degree. 
Hence every monomial in $\delta(T_{01})$ includes $T_{01}$ and therefore $T_{01}$ divides $\delta(T_{01})$. But $\delta(T_{01}) \ne 0$ by assumption. This contradicts Lemma~\ref{fr_lem}\ref{fr_lem_f}. 

2) $\rdeg \delta \le 0$. 
Then $\rdeg \delta(T_{02}) = \rdeg \delta + \rdeg T_{02} < 0$. 
In the same way, $T_{02}$ divides $\delta(T_{02})$ and 
this contradicts Lemma~\ref{fr_lem}\ref{fr_lem_f}.
\end{proof}

\section{Main results}

The following proposition and its proof are variations on~\cite[Theorem~4.4]{ArHaHeLi2012}. 
The distinctions are the following. We consider only the case of trinomial hypersurfaces, and in~\cite[Theorem~4.4]{ArHaHeLi2012} varieties defined by a system of trinomials are studied. 
At the same time, we need only the assumption that images of monomials are proportional, and do not use the primitivity of derivations. 
The latter condition is stronger, since according to~\cite[Proposition 3.5]{ArHaHeLi2012}, the dimension of the homogeneous component of degree~$w$ is equal to~$1$ whenever $w - \deg \g$ does not lie in the weight cone. 

\begin{proposition}
\label{prop_th}
Let $\delta \colon \RR \to \RR$ be a $\deg$-homogeneous locally nilpotent derivation. Suppose that $\delta(T_0^{l_0})$, $\delta(T_1^{l_1})$, and $\delta(T_2^{l_2})$ lie in a subspace of~$\RR$ of dimension~$1$. Then the derivation~$\delta$ is elementary. 
\end{proposition}

\begin{proof}
According to Lemma~\ref{3_lemma}, there is at most one variable $T_{ij}$ in each monomial $T_i^{l_i}$ such that $\delta(T_{ij}) \ne 0$. Let $\frK = \{i \mid \exists c_i \colon \delta(T_{ic_i}) \ne 0\}$.

Let us prove that if $\delta(T_{ic_i}) \ne 0$ and $\delta(T_{kc_k}) \ne 0$, then $l_{ic_i} = 1$ or $l_{kc_k} = 1$. 
Assume the converse, so $l_{ic_i} \ne 1$ and $l_{kc_k} \ne 1$. 
By the above, we obtain
\begin{equation}
\label{prop_th_1}
\delta(T_i^{l_i}) = \delta(T_{ic_i}) \, \frac{\partial T_i^{l_i}}{\partial T_{ic_i}}. 
\end{equation}
Since $l_{ic_i} \ne 1$, we have $T_{ic_i} \divid \partial T_i^{l_i} / \partial T_{ic_i}$, therefore $T_{ic_i} \divid \delta(T_{i}^{l_i})$. 
Similarly, $T_{kc_k} \divid \delta(T_k^{l_k})$.
By assumption, $\delta(T_i^{l_i})$ and $\delta(T_k^{l_k})$ are proportional, 
hence $T_{kc_k} \divid \delta(T_{i}^{l_i})$ and ${T_{ic_i} \divid \delta(T_k^{l_k})}$. 
At the same time, according to~(\ref{prop_th_1}),  
$\delta(T_i^{l_i})$ is a product of $\delta(T_{ic_i})$ and variables that are not equal to $T_{kc_k}$, whence by Lemma~\ref{div_lem}\ref{fact_graded} we have $T_{kc_k} \divid \delta(T_{ic_i})$. 
For the same reason, $T_{ic_i} \divid \delta(T_{kc_k})$, 
which contradicts Lemma~\ref{fr_lem}\ref{fr_lem_fg}. 

Thus $l_{i c_i} > 1$ holds for at most one~$i \in \frK$. 

The elements $\delta(T_i^{l_i})$ are proportional, that is there exists an~$f \in \RR$ such that \hbox{$\delta(T_i^{l_i}) \in \K f$} for any~$i$. 
It follows from equality~(\ref{prop_th_1}) that $\partial T_i^{l_i} / \partial T_{ic_i}$ divides $f$ for all $i \in \frK$. 
The application of Lemma~\ref{div_lem}\ref{fact_graded} yields that the product $\prod \limits_{i \in \frK} \partial T_i^{l_i} / \partial T_{ic_i}$ divides $f$. 
Denote this ratio by~$h$. 
From~(\ref{prop_th_1}) we have 
$$
f = h \prod \limits_{i \in \frK} \frac{\partial T_i^{l_i}}{\partial T_{ic_i}}, \quad
\delta(T_{ic_i}) = 
\begin{cases}
\beta_i h \prod \limits_{k \in \frK \setminus \{i\}} \frac{\partial T_k^{l_k}}{\partial T_{kc_k}}, &i \in \frK, \\
0, & i \notin \frK.
\end{cases}
$$

Let $\beta_i = 0$ for all $i \notin \frK$. Let us show that the sum of all $\beta_i$ is equal to $0$. Note that $\g$ should divide 
$$
\delta(\g) = \sum \limits_{i \in \frK} \delta(T_{ic_i}) \, \frac{\partial T_i^{l_i}}{\partial T_{ic_i}} = 
\sum \limits_{i \in \frK} 
\left(
	\beta_i h \prod \limits_{k \in \frK \setminus \{i\}} \frac{\partial T_k^{l_k}}{\partial T_{kc_k}}
\right) 
\frac{\partial T_i^{l_i}}{\partial T_{ic_i}} 
= h \left(\sum \limits_{i \in \frK} \beta_i\right) 
\prod \limits_{k \in \frK} \frac{\partial T_k^{l_k}}{\partial T_{kc_k}}
$$
in the algebra $\K[T_{ij}]$, 
since $\delta$ is a well-defined derivation of the algebra~$\RR$. 
But the trinomial~$\g$ and the monomial $\prod \limits_{k \in \frK} \left(\partial T_k^{l_k} / \partial T_{kc_k}\right)$ are coprime in the  factorial algebra~ $\K[T_{ij}]$. 
Hence $\g$ divides the sum $\sum \limits_{i \in \frK} \beta_i$, whence $\sum \limits_{i \in \frK} \beta_i = \sum \limits_{i} \beta_i = 0$. 

Now let $C = (c_0, c_1, c_2)$ be any sequence completing the $c_i$, where $i \in \frK$. Then we have $\delta = h \dCb$. The fact that $h$ belongs to the kernel of~$\dCb$ follows from Lemma~\ref{fr_lem}\ref{fr_lem_ker}. 
\end{proof}

We come to the main result of the paper. 

\begin{theorem}
\label{th_01m}
Let $\RR$ be a trinomial algebra. Suppose that there is at most one monomial in~$\g$ including a variable with exponent~$1$, that is $(l_{i_1j_1} = l_{i_2 j_2} = 1) \Rightarrow (i_1 = i_2)$. 
Then every $\deg$-homogeneous locally nilpotent derivation of the algebra~$\RR$ is elementary of Type~II. 
\end{theorem}

\begin{proof}
Let us denote a $\deg$-homogeneous locally nilpotent derivation of the algebra~$\RR$ by~$\delta$. 
By~Lemma~\ref{3_lemma}, there is at most one variable $T_{ij}$ in every monomial $T_i^{l_i}$ such that $\delta(T_{ij}) \ne 0$. 
Let $\frK = \{i \mid \exists c_i \colon \delta(T_{ic_i}) \ne 0\}$. 

1) Fix any $i \in \frK$ such that $\lici > 1$. 
We claim that $T_{ic_i}$ divides $\delta(T_{kc_k})$ for any $k \in \frK \setminus \{i\}$. 

Denote by $\Z_m = \{[0]_m, [1]_m, \ldots, [m-1]_m\}$ the cyclic group of order~$m$. 
Consider the following $\Z_\lici\!\!$-grading on~$\RR$: 
\begin{center}
\renewcommand{\arraystretch}{1.4}
\begin{tabular}{c ||    c|c|    c     |c|   c    |c}
$T_{ks}$ & 
$T_{01}$ & $T_{02}$ & 
$\ldots$ & $T_{ic_i}$ & $\ldots$ &  
$T_{2n_2}$ \\
\hline
$\rdeg T_{ks}$ & 
$[0]_\lici$ & $[0]_\lici$ & 
$\ldots$ & $[1]_\lici$ & $\ldots$ & 
$[0]_\lici$
\end{tabular}
\end{center}
It is well defined since $\lici [1]_\lici = [0]_\lici$, i.e., all monomials in~$\g$ have the same degree~$[0]_\lici$. 
By Lemma~\ref{deg_lem}, the derivation $\delta$ is $\rdeg$-homogeneous. 

If $\rdeg \delta(T_{ic_i}) \ne [0]_\lici$ then $T_{ic_i}$ divides $\delta(T_{ic_i})$. 
Indeed, this inequality implies that the degrees of all monomials in $\delta(T_{ic_i})$ do not equal~$[0]_\lici$, 
and it is possible only if every monomial includes $T_{ic_i}$ 
(since $T_{ic_i}$ is a unique variable with a nonzero degree). 

On the other hand, the fact that $T_{ic_i}$ divides $\delta(T_{ic_i})$ together with $\delta(T_{ic_i}) \ne 0$ contradicts Lemma~\ref{fr_lem}\ref{fr_lem_f}. 
Thus $\rdeg \delta(T_{ic_i}) = [0]_\lici$. 
Whence $\rdeg \delta = [\lici - 1]_\lici \ne [0]_\lici$. 
In particular, $\rdeg \delta(T_{kc_k}) \ne [0]_\lici$ for $k \in \frK \setminus \{i\}$, 
and therefore $T_{ic_i}$ divides $\delta(T_{kc_k})$. 

2) Let us show that $|\frK| \le 2$. 
Assume the converse. 
By assumptions of the theorem,  
$|\{i \mid \lici = 1\}| \le 1$, whence $|\{i \in \frK\mid \lici > 1\}| \ge 2$. 
Take any distinct $i, k$ from the latter set. 
According to~1), $T_{ic_i}$ divides $\delta(T_{kc_k})$ and $T_{kc_k}$ divides $\delta(T_{ic_i})$, 
which contradicts Lemma~\ref{fr_lem}\ref{fr_lem_fg}. 

3) Thus $|\frK| \le 2$. Then there are at most two nonzero $\delta(T_i^{l_i})$. On the other hand, $\delta(T_0^{l_0})+\delta(T_1^{l_1}) + \delta(T_2^{l_2}) = 
\delta(\g) = 0$ in the algebra $\RR$. 
Consequently, $\delta(T_0^{l_0})$, $\delta(T_1^{l_1})$, and $\delta(T_2^{l_2})$ lie in a subspace of dimension~$1$. 
According to Proposition~\ref{prop_th}, the derivation~$\delta$ is elementary. In addition, there are at most two nonzero $\delta(T_i^{l_i})$. This implies that $\delta$ is elementary of Type~II. 
\end{proof}

If a trinomial contains a linear term (i.e., $n_il_{i1} = 1$ for some $i = 0, 1, 2$), then the corresponding trinomial hypersurface is isomorphic to the affine space~$\K^{n - 1}$. 
Further we assume that $n_il_{i1} > 1$ for all $i = 0, 1, 2$. 
According to~\cite[Theorem 1.1(ii)]{HaHe2013}, in this case the following statement holds. 

\begin{proposition}
\label{fact_free}
The following conditions are equivalent: 
\begin{enumerate}
\item the algebra $\RR$ is factorial; 
\item the group $K$ in Construction~\ref{gr_constr} is torsion free; 
\item the numbers $d_i := \gcd(l_{i1}, \ldots, l_{in_i})$ are pairwise coprime. 
\end{enumerate}
\end{proposition}

This proposition enables us to prove some consequences of Theorem~\ref{th_01m}. 

\begin{corollary}
\label{cor_1}
If an algebra $\RR$ is non-factorial, then any $\deg$-homogeneous locally nilpotent derivation of~$\RR$ is elementary of Type~II. 
\end{corollary}

\begin{proof}
Let $\RR$ be a trinomial algebra that does not satisfy the conditions of Theorem~\ref{th_01m}. It suffices to show that $\RR$ is factorial. 

By assumption, $\g$ has at least two monomials including variables with exponent~$1$. Therefore at least two of $d_0, d_1, d_2$ are equal to~$1$, whence the numbers $d_0, d_1, d_2$ are pairwise coprime. By Proposition~\ref{fact_free}, $\RR$ is factorial. 
\end{proof}

\begin{corollary}
\label{cor_2}
The algebra~$\RR$ admits a nonzero $\deg$-homogeneous locally nilpotent derivation if and only if $l_{ij} = 1$ for some $i = 0, 1, 2$, $j = 1, \ldots, n_i$. 
\end{corollary}
\begin{proof}
If the condition $l_{ij} = 1$ holds for some pair $(i, j)$, then there exists a derivation of the form $\dCb$ (see Construction~\ref{dCb_constr}, case~\ref{dCbconstr_2}). It is $\deg$-homogeneous and locally nilpotent. 

Let us prove the inverse implication by contradiction. 
Let $l_{ij} \ge 2$ for all $i = 0, 1, 2$, $j = 1, \ldots, n_i$. 
Together with Theorem~\ref{th_01m} this implies that all homogeneous locally nilpotent derivations of~$\RR$ are of the form $h \dCb$ (where $h \in \Ker \dCb$). 
However it follows from Construction~\ref{dCb_constr} that there is no derivation of the form $\dCb$ under the condition $l_{ij} \ge 2$ for all $i = 0, 1, 2$, $j = 1, \ldots, n_i$. This contradiction proves the corollary. 
\end{proof}

An affine variety is called \textit{rigid} if its algebra of regular functions admits no nonzero locally nilpotent derivation. Geometrically, this means that the variety admits no non-trivial $\G_a$-action. 
The automorphism group of a rigid trinomial variety is described in~\cite[Theorem~5.5]{ArGa2017}. 

Corollary~\ref{cor_2} enables us to obtain a new proof of the result proved earlier in~\cite[Theorem~1]{Ar2016}. 

\begin{corollary}
\label{cor_3}
A factorial trinomial hypersurface~$X(\g)$ is rigid if and only if 
$l_{ij} \ge 2$ for any $i = 0, 1, 2$, $j = 1, \ldots, n_i$.
\end{corollary}

\begin{proof}
If the condition $l_{ij} \ge 2$ does not hold for some $(i, j)$, then there exists a derivation of the form~$\dCb$ (see Construction~\ref{dCb_constr}, case~\ref{dCbconstr_2}). This implies the ''only if'' part. 

Let us prove the ''if'' part. According to~Corollary~\ref{cor_2}, there exists no nonzero $\deg$-homogeneous locally nilpotent derivation of the algebra~$\RR$. By Proposition~\ref{fact_free} together with factoriality we have that the group~$K$ is torsion free. Then by Lemma~\ref{fact_hom_der_lem} the nonexistence of nonzero $\deg$-homogeneous locally nilpotent derivations implies the nonexistence of any nonzero locally nilpotent derivations of~$\RR$. This means rigidity. 
\end{proof}

\begin{example}
\label{ex_th_1332}
Consider $\g = T_{01} T_{02}^3 + T_{11}^3 + T_{21}^2$ 
(see Examples~\ref{ex_gr_1332} and \ref{ex_dCb_1332}). 
This trinomial satisfies the conditions of Theorem~\ref{th_01m}, 
whence all $\deg$-homogeneous locally nilpotent derivations~$\delta$ of~$\RR$ 
are divided into two classes:
\begin{enumerate}
\item \label{ex_hdCb_1332}
$
\delta = 2h\beta_0 T_{21} \dfrac{\partial}{\partial T_{01}} - h\beta_0 T_{02}^3 \dfrac{\partial}{\partial T_{21}}$, 
where $h \in \Ker \dCb$, $C = (1, 1, 1)$, $\beta = (\beta_0, 0, -\beta_0)$;
\item
$\delta = 3h\beta_0 T_{11}^2 \dfrac{\partial}{\partial T_{01}} - h\beta_0 T_{02}^3 \dfrac{\partial}{\partial T_{11}}$,
where $h \in \Ker \dCb$, $C = (1, 1, 1)$, $\beta = (\beta_0, -\beta_0, 0)$.
\end{enumerate}

Some of them are not primitive, i.e., $\deg \delta = \deg h \dCb$ may lie in the weight cone~$\omega$ for some $h \in \Ker \dCb$.  
Let us consider an example. The weight monoid is generated by the vectors
$$
\deg T_{01} = \begin{pmatrix} -3 \\ 3 \end{pmatrix}, \quad
\deg T_{02} = \begin{pmatrix} 1 \\ 1 \end{pmatrix}, \quad
\deg T_{11} = \begin{pmatrix} 0 \\ 2 \end{pmatrix}, \quad
\deg T_{21} = \begin{pmatrix} 0 \\ 3 \end{pmatrix},
$$
whence the weight cone~$\omega$ is the angle $\{-u \le v \le u\}$, where $\deg = \begin{pmatrix} u \\ v \end{pmatrix}$ (see  Figure~\ref{fig_ex_th_1332}). 

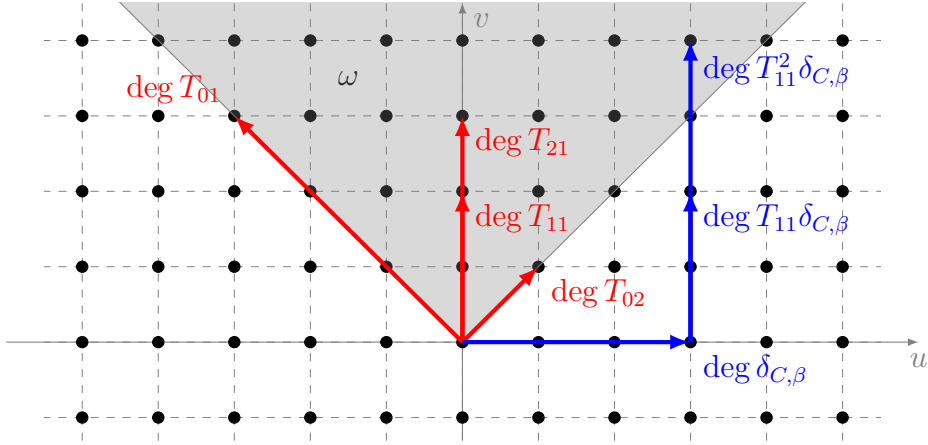
\begin{figure}[ht]
  \centering
  \begin{tikzpicture}
    \coordinate (Origin)   at (0,0);
    \coordinate (XAxisMin) at (-6,0);
    \coordinate (XAxisMax) at (6,0);
    \coordinate (YAxisMin) at (0,-1.3);
    \coordinate (YAxisMax) at (0,4.5);
    \draw [thin, gray,-latex] (XAxisMin) -- (XAxisMax) node [below] {$u$};
    \draw [thin, gray,-latex] (YAxisMin) -- (YAxisMax) node [below right] {$v$};

    \clip (-5.5,-1.3) rectangle (5.5cm,4.5cm); 
    \pgftransformcm{0.5}{0}{0}{0.5}{\pgfpoint{0cm}{0cm}}
    \coordinate (Bone) at (2,0);
    \coordinate (Btwo) at (0,2);
    \draw[style=help lines,dashed] (-14,-14) grid[step=2cm] (14,14);
    \foreach \x in {-7,-6,...,5}{
      \foreach \y in {-7,-6,...,4}{ 
        \node[draw,circle,inner sep=1.5pt,fill] at (2*\x,2*\y) {};
      }
    }
    \node at (-3,7) {$\omega$};
    \draw [thin,-latex,gray, fill=gray, fill opacity=0.3] 
         (0,0) -- ($-5*(Bone) + 5*(Btwo)$) -- ($5*(Bone) + 5*(Btwo)$) -- cycle;
    \draw [ultra thick,-latex,red] (Origin)
        -- ($-3*(Bone) + 3 *(Btwo)$) node [above left] {$\deg T_{01}$};
    \draw [ultra thick,-latex,red] (Origin)
        -- ($(Bone)+(Btwo)$) node [below right] {$\deg T_{02}$};
    \draw [ultra thick,-latex,red] (Origin)
        -- ($2*(Btwo)$) node [below right] {$\deg T_{11}$};
    \draw [ultra thick,-latex,red] (Origin)
        -- ($3*(Btwo)$) node [below right] {$\deg T_{21}$};
    \draw [ultra thick,-latex,blue] (Origin)
        -- ($3*(Bone)$) node [below right] {$\deg \dCb$};
    \draw [ultra thick,-latex,blue] ($3*(Bone)$)
        -- ($3*(Bone) + 2*(Btwo)$) node [below right] {$\deg T_{11} \dCb$};
    \draw [ultra thick,-latex,blue] ($3*(Bone)$)
        -- ($3*(Bone) + 4*(Btwo)$) node [below right] {$\deg T_{11}^2 \dCb$};
  \end{tikzpicture}
  \caption{The weight cone in Example~\ref{ex_th_1332}}
  \label{fig_ex_th_1332}
\end{figure}
In case~\ref{ex_hdCb_1332}, we have $\deg \dCb = \begin{pmatrix} 3 \\ 0 \end{pmatrix}$. 
Since for any $k \in \N$ the polynomial $h = T_{11}^k$ belongs to the kernel of~$\dCb$, it follows that 
the derivation $T_{11}^k \dCb$ is homogeneous and locally nilpotent. 
Its degree equals 
$k\begin{pmatrix} 0 \\ 2 \end{pmatrix} + \begin{pmatrix} 3 \\ 0 \end{pmatrix} = \begin{pmatrix} 3 \\ 2k \end{pmatrix}$ and 
lies in the weight cone for $k \ge 2$. 
\end{example}

\begin{example}
\label{ex_th_222}
Let $\g = T_{01}^2 + T_{11}^2 + T_{21}^2$ (see Examples~\ref{ex_gr_222} and \ref{ex_dCb_222}). 
According to Corollary~\ref{cor_2}, there exists no $\deg$-homogeneous locally nilpotent derivation of~$\RR$.

Nevertheless the algebra~$\RR$ admits a nonzero locally nilpotent derivation. For example, consider the following derivation~$\delta$: 
$$
\delta(T_{01}) = i T_{21}, \quad 
\delta(T_{11}) = -T_{21}, \quad 
\delta(T_{21}) = -i T_{01} + T_{11},
$$
where $i \in \K$, $i^2 = -1$. 
The derivation $\delta$ is locally nilpotent, since $\delta(-iT_{01}+T_{11}) = 0$, $\delta^2(T_{21}) = 0$, $\delta^3(T_{01}) = 0$. 
\end{example}

\section{Open questions}
\label{open}

In this section we discuss several open questions related to locally nilpotent derivations of trinomial algebras. 

\begin{conjecture}
\label{conj}
All $\deg$-homogeneous locally nilpotent derivations of a trinomial algebra are elementary. 
\end{conjecture}

In the proof of Theorem~\ref{th_01m} we apply Proposition~\ref{prop_th} to confirm Conjecture~\ref{conj} for trinomial algebras~$\RR$ corresponding to trinomials~$\g$ 
under the following condition: there exists at most one monomial in $\g$ including variables with exponent~$1$. 
We would like to apply Proposition~\ref{prop_th} in the remaining case too, that is, when there exist at least two monomials in~$\g$ including variables with exponent~$1$. 
In this case, by Proposition~\ref{fact_free}, the algebra~$\RR$ is factorial, the group~$K$ is torsion free, and the quasitorus, whose action corresponds to the $\deg$-grading, is a torus of dimension~$n - 2$. 
In~\cite{Ko2014_eng}, all $\deg$-homogeneous locally nilpotent derivations for several trinomial hypersurfaces are described. These hypersurfaces do not satisfy the conditions of Theorem~\ref{th_01m}. 
In particular, the following statements are proved in \cite[Theorem~3.22 and Theorem~3.24]{Ko2014_eng}. 

\begin{example}[compare with Example~\ref{ex_dCb_11112}]
Any $\deg$-homogeneous locally nilpotent derivation of the algebra~$\RR$, where 
$\g = T_{01}T_{02} + T_{11}T_{12} + T_{21}^2$, has the form 
\begin{gather*}
\lambda T_{0i}^k T_{1j}^l T_{21}^p 
\left(T_{1j} \dfrac{\partial}{\partial T_{1\bar i}} - T_{0i} \dfrac{\partial}{\partial T_{1\bar j}}\right) \quad\text{or} 
\\
T_{0i}^k T_{1j}^l (\alpha_1 T_{01}T_{02} - \alpha_0 T_{11}T_{12})^p 
\left(\alpha_0 T_{1j} T_{21} \dfrac{\partial}{\partial T_{0\bar i}} 
+ \alpha_1 T_{0i} T_{21} \dfrac{\partial}{\partial T_{1\bar j}}
- \dfrac{\alpha_0 + \alpha_1}{2} T_{0i} T_{1j} \dfrac{\partial}{\partial T_{21}}
\right), 
\end{gather*}
where $k, l, p \in \Zgezero$, $\{i, \bar i\} = \{j, \bar j\} = \{1, 2\}$,  $\alpha_0, \alpha_1, \lambda \in \K$, $\alpha_0 + \alpha_1 \ne 0$. 
\end{example}

\begin{example}
Any $\deg$-homogeneous locally nilpotent derivation of the algebra~$\RR$, where 
$\g = T_{01}T_{02} + T_{11}T_{12} + T_{21}T_{22}$, has the form 
\begin{gather*}
T_{0i_0}^{k_0} T_{1i_1}^{k_1} T_{2i_2}^{k_2} (\alpha_2 T_{11}T_{12} - \alpha_1 T_{21}T_{22})^p 
\left(\alpha_0 T_{1i_1} T_{2i_2} \dfrac{\partial}{\partial T_{0\bar i_0}} 
+ \alpha_1 T_{0i_0} T_{2i_2} \dfrac{\partial}{\partial T_{1\bar i_1}} 
+ \alpha_2 T_{0i_0} T_{1i_1} \dfrac{\partial}{\partial T_{2\bar i_2}} 
\right), 
\end{gather*}
where $k_0, k_1, k_2, p \in \Zgezero$, $\{i_0, \bar i_0\} = \{i_1, \bar i_1\} = \{i_2, \bar i_2\} = \{1, 2\}$,  
$\alpha_0, \alpha_1, \alpha_2 \in \K$, $(\alpha_0, \alpha_1, \alpha_2) \ne (0, 0, 0)$, and \hbox{$\alpha_0 + \alpha_1 + \alpha_2 = 0$}. 
\end{example}

One can check that any derivation~$\delta$ from these examples is elementary, so Conjecture~\ref{conj} holds in these cases. 
Moreover, $\delta(T_0^{l_0}), \delta(T_1^{l_1})$, and $\delta(T_2^{l_2})$ lie in a subspace of dimension~$1$, hence Conjecture~\ref{conj} can be confirmed in these cases by Proposition~\ref{prop_th} as well. 

\smallskip

Besides the ''finest'' $\deg$-grading on the algebra~$\RR$ by the group~$K$ one can consider another grading on~$\RR$ by the torsion free component~$K_0$ of~$K$. 
Notice that the condition to be homogeneous with respect to the $\deg$-grading by the group~$K$ is more restrictive than the condition to be homogeneous with respect to the grading by~$K_0$. 
According to Proposition~\ref{fact_free}, the $\deg$-grading by~$K$ coincides with the grading by~$K_0$ if and only if the algebra~$\RR$ is factorial. 

\begin{example}
Let $\g = T_{01}^2 + T_{11}^2 + T_{21}^2$ (see Examples~\ref{ex_gr_222}, \ref{ex_dCb_222}, and \ref{ex_th_222}). 
According to Example~\ref{ex_gr_222}, we have $K = \Z + \Z_2 + \Z_2$. Therefore $K_0 = \Z$ and the degree with respect to the grading by~$K_0$ is equal to the standard degree of a polynomial (the sum of exponents). 
Hence the derivation~$\delta$ from Example~\ref{ex_th_222} is homogeneous with respect to the grading by~$K_0$. 
At the same time, according to Example~\ref{ex_th_222}, there is no nonzero locally nilpotent derivation that is homogeneous with respect to the $K$-grading. 
\end{example}

If the algebra~$\RR$ admits no nonzero $K_0$-homogeneous locally nilpotent derivation, then by Lemma~\ref{fact_hom_der_lem} $\RR$ admits no nonzero locally nilpotent derivation, i.e., the variety~$X(\g)$ is rigid. 

\begin{problem}
\label{quest_K0}
Describe all locally nilpotent derivations of a trinomial algebra~$\RR$ that are homogeneous with respect to the $K_0$-grading. 
\end{problem}

Whereas the $K$-grading corresponds geometrically to the action of a quasitorus on the trinomial hypersurface, 
the $K_0$-grading corresponds to the action of its neutral component, i.e., of the maximal torus. 
This means that Problem~\ref{quest_K0} asks for a description of $\G_a$-actions normalized by the maximal torus, 
that is, of the root subgroups in the automorphism group of a trinomial hypersurface. 

\begin{problem}
Describe all (not necessary homogeneous) locally nilpotent derivations of a trinomial algebra~$\RR$. 
\end{problem}

\end{document}